\documentclass[10pt,reqno]{amsart}
\usepackage{amsmath,amsopn,amssymb,amsthm}
\usepackage[cp1251]{inputenc}
\usepackage[english]{babel}
\usepackage[pagebackref,breaklinks=true,colorlinks=true,linkcolor=blue,citecolor=blue,urlcolor=blue]{hyperref}
\usepackage{graphicx}%
\usepackage{color}

\newtheorem{theorem}{Theorem}[section]

\newtheorem{corollary}[theorem]{Corollary}

\newtheorem{lemma}[theorem]{Lemma}

\newtheorem{remark}[theorem]{Remark}

\renewenvironment{proof}[1][Proof]{\noindent\textbf{#1.} }{\ \rule{0.5em}{0.5em}}

\begin{document}
\title[Left invariant spray structure on Lie group]{Left invariant spray structure on Lie group}
\author{Ming Xu}

\address{Ming Xu \newline
School of Mathematical Sciences,
Capital Normal University,
Beijing 100048,
P. R. China}
\email{mgmgmgxu@163.com}

\date{}

\begin{abstract}
We use the technique of invariant frame to study the left invariant
spray structure on a Lie group, and calculate its S-curvature and Riemann curvature, which generalizes the corresponding formulae in  homogeneous Finsler geometry. Using the canonical bi-invariant Berwald spray structure as the reference, any left invariant spray structure can be associated with a spray vector field on the Lie algebra. We find the correspondence between the geodesics for a left invariant spray structure and the inverse integral curves of its spray vector field. As an application for this correspondence, we provide an alternative proof for Landsberg Conjecture in the case of homogeneous Finsler surfaces.


\end{abstract}
\maketitle

\section{Introduction}

On a Finsler manifold $(M,F)$, every geodesic with positive constant speed can be lifted to an integral curve for the smooth tangent vector field $\mathbf{G}_F$ on $TM\backslash0$ with the standard local coordinate presentation $$\mathbf{G}_F=y^i\partial_{x^i}-2\mathbf{G}_F^{i}\partial_{y^i},
\ \
\mbox{in which}\ \
\mathbf{G}_F^i=\tfrac14 g^{il}([F^2]_{x^ky^l}y^k-[F^2]_{x^l}).$$
This $\mathbf{G}_F$, which may be called a {\it Finsler spray structure} on $M$, is crucial for studying Finsler geometry. Many notions, like geodesic, Riemann curvature, etc, may be defined from the spray and without direct appearance of the metric \cite{BCS2000}, so they can be studied in more general context \cite{Be1926}. This observation brings us {\it spray geometry}.

In spray geometry, a manifold $M$ is only endowed with a spray structure, i.e., a smooth tangent vector field $\mathbf{G}$ on $TM\backslash0$ satisfying certain mild requirement. The purpose of imposing the spray structure is to specify the set of all geodesics (with required parametrizations) on $M$ \cite{Sh2001}. See Section 2.2 and Section 2.3 for the details and related notions. There are many spray structures which can not be realized by Finsler metrics \cite{LS2017,Ya2011}. They exhibit for us interesting dynamic or geometric phenomena which are unseen in Finsler geometry.

Lie method can be applied to spray geometry. A manifold $M$ with a spray structure $\mathbf{G}$ may be called {\it homogeneous} (or {\it affinely homogeneous}),
if it admits a smooth transitive Lie group action which preserves $\mathbf{G}$. Here the spray preserving property means that  the lifting (or tangent map) $\widetilde{\rho}=\rho_*:TM\rightarrow TM$ for the diffeomorphism $\rho$ on $M$ satisfies $(\widetilde{\rho})_*\mathbf{G}=\mathbf{G}$, or equivalently, both $\rho$ and $\rho^{-1}$ maps geodesics to geodesics. In the Lie algebraic level, the spray preserving property
for a flow $\rho_t$ of diffeomorphisms on $M$ generated by a   tangent  field $V$ is then encoded in $[\widetilde{V},\mathbf{G}]=0$, where $\widetilde{V}$ is the lifting of $V$ on $TM\backslash0$ (See Section \ref{section-2-1} or \cite{XD2015} for more details on lifting tangent fields).

From the view point of homogeneous geometry, left invariant geometry
on a Lie group is an important model which deserves to be singled out for study \cite{DZ1979,Mi1976,YD2017}. On a Lie group $G$, we have globally defined frames consisting of left and right invariant tangent vector fields respectively. We further expand them to the left invariant frame $\{\widetilde{U}_i,\partial_{u^i},\forall i\}$ and right invariant frame
$\{\widetilde{V}_i,\partial_{v^i},\forall i\}$ on $TG$. Take  $\{\widetilde{U}_i,\partial_{u^i},\forall i\}$ for example. Here $U_i$'s are left invariant
tangent vector fields on $G$  and $\widetilde{U}_i$'s are their liftings.
At each $g\in G$, the values $\{U_i(g),\forall i\}$ provide a basis for $T_gG$.
The vector fields
$\partial_{u^i}$'s are tangent to each $T_gG$, where they correspond to the linear coordinates for $y=u^iU_i(g)\in T_gG$.

We show in this paper that, using these invariant frames, the geometry for a left invariant spray structure can be explicitly and globally described. Firstly, we prove the existence of the most special bi-invariant spray structure on each Lie group, i.e., (see its proof in Section \ref{section-3-2})\bigskip

\noindent
{\bf Theorem A}\quad {\it The vector field
$\mathbf{G}_{0}=u^i\widetilde{U}_i=v^i\widetilde{V}_i$ on $TG\backslash0$
is a bi-invariant Berwald spray structure on the Lie group $G$.}
\bigskip

Then we use $\mathbf{G}_0$ as the reference to present any other left invariant spray structure $\mathbf{G}$, i.e., $\mathbf{G}=\mathbf{G}_0-\mathbf{H}$. Here  $\mathbf{H}=\mathbf{H}^i\partial_{u^i}$ is a left invariant smooth vector field on $TG\backslash0$, where the left invariancy is with respect to the action of $\widetilde{L}_g$ for
every $g\in G$.

Collecting and combining the data from $\mathbf{H}$, from the left invariant frame, and from the Lie bracket coefficients $c_{ij}^k$ for $\mathfrak{g}=\mathrm{Lie}(G)$,
we find global presentations for  curvatures of $(G,\mathbf{G})$.
For the S-curvature, we have (see Section \ref{section-3-3} for its proof)\bigskip

\noindent
{\bf Theorem B}\quad
{\it For the left invariant spray structure
$\mathbf{G}=\mathbf{G}_0-\mathbf{H}^i\partial_{u^i}$
and any left invariant smooth measure on $G$,
its S-curvature can be presented as
$
S=\tfrac12 \tfrac{\partial}{\partial u^i}\mathbf{H}^i
+\tfrac12 c_{lj}^j u^l$.}\bigskip

\noindent
For the Riemann curvature, we have (see Section \ref{section-3-4} for its proof)\bigskip

\noindent
{\bf Theorem C}\quad
{\it For the left invariant spray structure
$\mathbf{G}=\mathbf{G}_0-\mathbf{H}^i\partial_{u^i}$,
the Riemann curvature satisfies
\begin{eqnarray*}
R(\widetilde{U}^\mathcal{H}_q)&=&\tfrac34 c^r_{pq}u^p
\tfrac{\partial}{\partial u^r}\mathbf{H}^i\partial_{u^i}+
\tfrac12c^i_{qj}\mathbf{H}^j\partial_{u^i}+
\tfrac12\mathbf{H}^p\tfrac{\partial^2}{\partial u^p\partial u^q}
\mathbf{H}^i\partial_{u^i}\nonumber\\& &-
\tfrac14\tfrac{\partial}{\partial u^q}
\mathbf{H}^p\tfrac{\partial}{\partial u^p}\mathbf{H}^i\partial_{u^i}+
\tfrac14 c^i_{pr}u^r\tfrac{\partial}{\partial u^q}\mathbf{H}^p
\partial_{u^i}-
\tfrac14 c^p_{qj}c^i_{pr}u^ju^r\partial_{u^i},
\end{eqnarray*}
where $\widetilde{U}_q^\mathcal{H}$ is the horizonal lifting of $U_q$}.\bigskip

\noindent

On the other hand, the algebraization method tells us, we can always reduce a problem in homogeneous geometry to that in a tangent space.
For the left invariant spray structure $\mathbf{G}=\mathbf{G}_0-\mathbf{H}$,  the restriction $\mathbf{H}|_{T_eG\backslash\{0\}}$ plays an important role.
It is more convenient to present it as the {\it spray vector field} (this terminology is from L. Huang \cite{Hu2015-1}), i.e., the map $\eta:\mathfrak{g}\backslash\{0\}\rightarrow\mathfrak{g}$ determined by
$\eta(y)=\mathbf{H}^i(e,y) e_i$  with $y=u^ie_i=u^i U_i(e)\in\mathfrak{g}$.
Then Theorem B and Theorem C can be easily translated to
Corollary \ref{cor-4}, which generalizes the left invariant curvature formulae L. Huang found in
homogeneous Finsler geometry \cite{Hu2015-1,Hu2015-2,Hu2017}.

The importance of the spray vector field $\eta$ is further revealed by
the correspondence between geodesics for $\mathbf{G}$ and integral curves of $-\eta$. To be precise, we prove the following theorem in Section \ref{section-4-2}.\bigskip

\noindent
{\bf Theorem D}\quad
{\it
Let $\mathbf{G}$ be the left invariant spray structure on the Lie group $G$ with the associated spray vector field $\eta$. Then for any open interval $(a,b)\subset\mathbb{R}$ containing 0, there is a
one-to-one correspondence between the following two sets:
\begin{enumerate}
\item The set of all $c(t)$ with $t\in(a,b)$ and $c(0)=e$, which are geodesics for $\mathbf{G}$;
\item The set of all $y(t)$ with $t\in(a,b)$, which are integral curves of $-\eta$.
\end{enumerate}
The correspondence between these two sets is given by $y(t)=(L_{c(t)^{-1}})_*(\dot{c}(t))$.
}\bigskip

There is an interesting application. Recently,
B. Najafi and A. Tayebi proved Landsberg Conjecture \cite{Ma1996}
(see \cite{XM2020} for more references and some recent progress)
in the case of homogeneous Finsler surfaces. We can use Theorem D to
propose a totally
different proof for their theorem in \cite{NT2021} (see Theorem \ref{thm-4} and its proof).

Finally, we remark that the geometry of a left invariant spray structure
has many similarities as its analog in Finsler geometry, but there are also many significant differences. We believe that many results for left invariant spray structure can be generalized to homogeneous spray structures on other smooth coset spaces, however they may lose their globality after the generalization, and the proofs would require a more delicate usage of special local frame.

This paper is organized as following. In Section 2, we recall the notion for lifting a vector field and basic knowledge on spray geometry. In Section 3, we introduce invariant spray structures on a Lie group and prove the left invariant formulae for S-curvature and Riemann curvature. In Section 4, we discuss the spray vector field $\eta$ for a left invariant spray structure $\mathbf{G}$, prove the correspondence between geodesics for $\mathbf{G}$ and integral curves of $-\eta$,  and provide an alternative proof for B. Najafi and A. Tayebi's theorem.

\section{Preliminaries}
\subsection{Lifting of a tangent field}
\label{section-2-1}
 Throughout this paper, the smoothness for manifolds, fields, measures, and curves, etc, and the connectedness for Lie groups, are always assumed.

A tangent field $V$ on the manifold $M$ can be lifted to tangent field $\widetilde{V}$ on $TM$. Locally around any $x\in M$,
$V$ generates a flow of local diffeomorphisms $\rho_t$ on $M$. We can first lift $\rho_t$ to $\widetilde{\rho}_t=(\rho_t)_*$ on $TM$, and then define the {\it lifting} (or {\it complete lifting}) of $V$ by $\widetilde{V}=\tfrac{d}{dt}|_{t=0}\widetilde{\rho}_t$.

We may use {\it standard local coordinates} $(x^i,y^i)$, i.e., $x=(x^i)\in M$ and $y=y^i{\partial}_{x^i}$, to present $\widetilde{V}$, i.e., (see Lemma 3.2 in \cite{XD2015})
\begin{lemma}\label{lemma-1}
The tangent field $V(x)=a^i(x) \partial_{x^i}$ has the lifting
\begin{equation}\label{001}
\widetilde{V}=a^i(x)  \partial_{x^i}+
y^j\tfrac{\partial}{\partial x^j}a^i(x)  \partial_{y^i}
\end{equation}
by standard local coordinates.
\end{lemma}

Notice that in the local frame $\{\partial_{x^i},\partial_{y^i},\forall i\}$ corresponding to the standard local coordinates $(x^i,y^i)$, each $\partial_{x^i}$ on $TM$ coincides with the lifting of $\partial_{x^i}$ on $M$, and each $\partial_{y^i}$ is tangent to every tangent space.
So Lemma \ref{lemma-1} implies the following transfer formulae for  changing the standard local coordinates around a point,
\begin{eqnarray*}
& &\partial_{\bar{x}^i}=\tfrac{\partial{x}^j}{\partial \bar{x}^i}  \partial_{{x}^j}+
{y}^j\tfrac{\partial}{\partial x^j}(\tfrac{\partial x^k}{\partial\bar{x}^i})  \partial_{y^k}=
\tfrac{\partial{x}^j}{\partial \bar{x}^i}  \partial_{{x}^j}+
\bar{y}^j\tfrac{\partial^2 x^k}{\partial\bar{x}^i\partial\bar{x}^j}  \partial_{y^k}\label{002},\\
& &\bar{y}^i=y^j\tfrac{\partial \bar{x}^i}{\partial x^j} \quad\mbox{and}\quad
{\partial_{\bar{y}^i}}=\tfrac{\partial{x}^j}{\partial \bar{x}^i} \partial_{y^j}.\label{003}
\end{eqnarray*}
Though geometric notions are sometimes introduced by standard local coordinates, we can use above transfer formulae to verify they are well defined.

A {\it local frame} is
referred to a set of tangent fields $\{U_i,\forall i\}$ on some open subset $\mathcal{U}$
of $M$, such that
at each $x\in\mathcal{U}$, $\{U_i(x),\forall i\}$ is a basis for $T_xM$. Associated with $\{U_i,\forall i\}$, there are functions $u^i$'s on $T\mathcal{U}$ determined by $y=u^i  U_i(x)$ for every $y\in T_xM\subset T\mathcal{U}$. We denote $\partial_{u^i}$'s
the tangent fields which are tangent to and correspond to the linear $u^i$-coordinates in each $T_xM\subset T\mathcal{U}$. Using the local frame $\{\widetilde{U}_i,\partial_{u^i},\forall i\}$ on $T\mathcal{U}$, Lemma \ref{lemma-1} can be slightly generalized
as following.

\begin{lemma}\label{lemma-3}
Let $V(x)=a^i(x)  U_i$ be the local frame presentation for a tangent field on $M$.
Then its lifting can be locally presented as
\begin{equation*}
\widetilde{V}=a^i \widetilde{U}_i+u^j U_j a^i   \partial_{u^i}.
\end{equation*}
\end{lemma}
\begin{proof}
Using standard local coordinates $(x^i,y^i)$, we have the presentations
\begin{eqnarray}\label{006}
U_i=A_i^j \partial_{x^j},\quad  u^i=y^j B^i_j,\quad\mbox{and}\quad
\partial_{u^i}=A^j_i \partial_{y^j},
\end{eqnarray}
where $(A_i^j)=(A_i^j(x))$ and $(B_i^j)=(B_i^j(x))=(A_i^j(x))^{-1}$ (i.e. $A_i^j B_j^k=B_i^jA_j^k=\delta_i^k$) are matrix valued functions which only depend on the $x^i$-coordinates. By Lemma \ref{lemma-1},
the lifting of $V=a^i  U_i=a^iA_i^j \partial_{x^j}$ is
\begin{eqnarray*}
\widetilde{V}&=&a^i A_i^j  \partial_{x^j}+y^j\tfrac{\partial}{\partial x^j}
(a^iA_i^k)  \partial_{y^k}\\
&=& a^i(A_i^j  \partial_{x^j}+y^j\tfrac{\partial}{\partial x^j}A_i^k  \partial_{y^k})+ y^j A_i^k\tfrac{\partial}{\partial x^j}a^i  \partial_{y^k}\\
&=&a^i \widetilde{U}_i+ u^j U_j a^i  \partial_{u^i},
\end{eqnarray*}which ends the proof.
\end{proof}

By straight forward calculation using Lemma \ref{lemma-1}, we see

\begin{lemma} \label{lemma-2}
(1) For the diffeomorphism $\phi$, tangent fields $V$ and
$\phi_*V$ on $M$, their liftings satisfy $\widetilde{\phi}_*\widetilde{V}=\widetilde{\phi_*V}$. So if $\phi$
preserves $V$, i.e., $\phi_*V=V$, then $\widetilde{\phi}$ preserves $\widetilde{V}$ as well.

(2) For the liftings $\widetilde{U}$, $\widetilde{V}$ and $\widetilde{[U,V]}$ of the tangent fields $U$, $V$ and $[U,V]$ on $M$, we have $\widetilde{[U,V]}=[\widetilde{U},\widetilde{V}]$. In particular,
$\widetilde{U}$ and $\widetilde{V}$ commute when $U$ and $V$ do.
\end{lemma}

Since lifting diffeomorphisms from $M$ to $TM$ might be viewed as
a Lie group homomorphism, and lifting tangent vector fields the corresponding Lie algebra homomorphism, Lemma \ref{lemma-2} can be naturally observed.

\subsection{Spray structure and geodesic}
\label{section-2-2}
Here we collection some basic knowledge on the spray geometry from
\cite{Sh2001}, which are required in later discussion.

Let $\mathbf{G}$ be a tangent field on the slit tangent bundle
$TM\backslash0$ for the manifold $M$. We call it a {\it spray structure} on $M$, if for every standard local coordinates $(x^i,y^i)$, $\mathbf{G}$ has a presentation of the form
$$\mathbf{G}=y^i \partial_{x^i}-2\mathbf{G}^i \partial_{y^i},$$
where each $\mathbf{G}^i=\mathbf{G}^i(x,y)$ has the degree-two positive homogeneity with respect to the $y^i$-coordinates \cite{Sh2001}. Further more, if all $\mathbf{G}^i(x,\cdot)$'s are  quadratic for the $y$-coordinates, we call $\mathbf{G}$ a {\it Berwald spray structure}.

With respect to a measure which is locally presented
as $d\mu=\sigma(x)dx^1\wedge\cdots\wedge dx^n$, where $\sigma(x)$
is nonvanishing everywhere in the local chart, the S-curvature \cite{Sh1997}
can be determined by its local presentation
\begin{equation}\label{118}
S(x)=\tfrac{\partial}{\partial y^i}\mathbf{G}^i(x,y)-\sigma(x)^{-1}y^m
\tfrac{\partial}{\partial x^m}\sigma(x).
\end{equation}

The spray structure one-to-one determines the set of all the geodesics on $M$ (with required parametrization). A curve $c(t)$ on $(M,\mathbf{G})$ is called a {\it geodesic} if its tangent field $\dot{c}(t)$ is nonzero everywhere and its lifting $(c(t),\dot{c}(t))$ is an integral curve of  $\mathbf{G}$. By standard local coordinates, a geodesic $c(t)=(c^i(t))$ is determined by the ODEs
$$\ddot{c}^i(t)+2\mathbf{G}^i(c(t),\dot{c}(t))=0,\quad\forall i.$$
We say $\mathbf{G}$ is {\it geodesically complete}, if every maximally extended geodesic $c(t)$ corresponds to $t\in(-\infty,\infty)$.

\subsection{Riemann curvature for a spray structure}
\label{section-2-3}
For the spray structure $\mathbf{G}$ on $M$ and any standard local coordinates $(x^i,y^i)$, we denote $N^i_j=\tfrac{\partial {\mathbf{G}}^i}{\partial y^j}$ and
$\delta_{x^i}=\partial_{x^i}-N^j_i\partial_{y^j}$. Then the tangent bundle of $TM\backslash0$ is the direct sum
of two subbundles. One is the {\it horizonal distribution}
$\mathcal{H}$ linearly spanned by all
${\delta_{x^i}}$'s at each point. The other is
the {\it vertical distribution} $\mathcal{V}$ linearly spanned by all ${\partial_{ y^i}}$'s at each point. For a tangent field $V=a^i(x)\partial_{x^i}$, we call  $\widetilde{V}^\mathcal{H}=a^i(x)\delta_{x^i}$ the {\it horizonal lifting} of $V$.

For any nonzero $y\in T_xM$, the Riemannian curvature $\mathbf{R}_y:T_xM\rightarrow T_xM$
can be presented as the linear map $\mathbf{R}_y(a^k\partial_{x^k})=a^kR^i_k\partial_{x^i}$, in which
the coefficients $R_k^i$'s are determined by
\begin{equation}\label{004}
[\mathbf{G},\delta_{x^k}]\equiv R^i_k\partial_{y^i}\quad\mbox{(mod }
\mathcal{H}\mbox{)}.
\end{equation}
Using all $R^i_k$'s in (\ref{004}), the Riemann curvature can also be interpreted as a tensor field $R=R^i_k d x^k\otimes {\partial}_{y^i}$ for $\mathcal{H}^*\otimes\mathcal{V}$ over $TM\backslash0$.
\section{Invariant spray structure on a Lie group}
\subsection{Notations for invariant frame}
\label{section-3-1}
Let $G$ be a  Lie group. We denote $L_g(g')=gg'$ and $R_g(g')=g'g$ its left and right translations. Let $\mathfrak{g}=T_eG$ be the Lie algebra of $G$, for which
we fix a basis $\{e_1,\cdots,e_n\}$, and denote $c_{ij}^k$ the corresponding bracket coefficients in $[e_i, e_j]=c_{ij}^k e_k$.

Each $e_i$ determines a left invariant tangent field $U_i(g)=(L_g)_*(e_i)$. Any tangent vector $y\in T_gG$ can be uniquely written as $y=u^iU_i(g)$.
Let $\partial_{u^i}$'s be the sections of $\mathcal{V}$ which correspond to the $u^i$-coordinates in each $T_gG$. We will simply call $\{\widetilde{U}_i,\partial_{u^i},\forall i\}$ the {\it left invariant frame on $TG$}.

Similarly, we also have the right invariant tangent fields
$V_i(g)=(R_g)_*(e_i)$, the functions $v^i$ on $TG$ determined by
$y=v^iV_i(g)$ for every $y\in T_gG$, and the {\it right invariant frame $\{\widetilde{V}_i,\partial_{v^i},\forall i\}$
on $TG$}. By their invariancy, we have the following
obvious facts for every $i$ and $j$,
\begin{eqnarray}
& &[U_i,U_j]= c_{ij}^kU_k, \quad [V_i,V_j]= -c_{ij}^k V_k,\quad
[U_i,V_j]=0,\nonumber\\
& &[\widetilde{U}_i,\widetilde{U}_j]= c_{ij}^k\widetilde{U}_k, \quad [\widetilde{V}_i,\widetilde{V}_j]= -c_{ij}^k V_k,\quad
[\widetilde{U}_i,\widetilde{V}_j]=0,\nonumber\\
& &\widetilde{U}_i v^j=0,\quad [\widetilde{U}_i,\partial_{v^j}]=0,
\quad\widetilde{V}_i u^j=0,\quad [\widetilde{V}_i,\partial_{u^j}]=0.
\label{100}
\end{eqnarray}

Denote $\phi_i^j$ and $\psi_i^j$ the functions on $G$ such that
$\mathrm{Ad}(g)e_i=\phi_i^j e_j$, $\mathrm{Ad}(g^{-1})e_i=\psi_i^j e_j$ (so we have $(\psi_i^j)=(\phi_i^j)^{-1}$, i.e., $\psi_i^j\phi_j^k=\phi_i^j\psi_j^k=\delta_i^k$). At each $g\in G$,
\begin{eqnarray*}
U_i(g)=(L_g)_* (e_i)=(R_{g})_*(R_{g^{-1}})_*(L_{g})_*( e_i)
=(R_{g})_*(\mathrm{Ad}(g)e_i)=\phi_i^j (R_{g})_* (e_j)=
\phi_i^j V_j(g).
\end{eqnarray*}
So we have
\begin{eqnarray}\label{101}
U_i=\phi_i^j  V_j,\quad
u^i=\psi^i_j v^j,\quad \partial_{u^i}=\phi_i^j \partial_{v^j}.
\end{eqnarray}

In later discussion, we need the following results.
\begin{lemma}\label{lemma-5}
Keep above notations, then at each point of $G$, we have the following:
\begin{eqnarray*}
(1)\ \phi_l^j V_j\phi_i^k=c^j_{li}\phi_j^k,\quad
(2)\ \widetilde{U}_i=\phi_i^j\widetilde{V}_j+
c^q_{pi}u^p\partial_{u^q},\quad
(3)\ \widetilde{U}_i u^j=c^j_{li}u^l,\quad
(4)\ [\widetilde{U}_i,{\partial_{u^l}}]
=c^p_{il}\partial_{u^p}.
\end{eqnarray*}
\end{lemma}
\begin{proof}
(1) At each $g\in G$, we have $\mathrm{Ad}(g)e_i=\phi_i^k  e_k$
and then
\begin{eqnarray*}
V_j\phi_i^k  e_k&=& V_j(\phi_i^k  e_k) =\tfrac{d}{dt}
(\mathrm{Ad}(\exp te_j  \cdot g)  e_i)\\
&=&\tfrac{d}{dt}\mathrm{Ad}(\exp te_j)(\mathrm{Ad}(g)e_i)
=[e_j,\mathrm{Ad}(g)e_i].
\end{eqnarray*}
So we get
\begin{eqnarray}\label{007}
\phi_l^jV_j\phi_i^k  e_k=[\mathrm{Ad}(g)e_l,\mathrm{Ad}(g)e_i]
=\mathrm{Ad}(g)[e_l,e_i]=c_{li}^j\mathrm{Ad}(g)e_j=c_{li}^j\phi_j^k e_k.
\end{eqnarray}
Comparing the coefficients of $e_k$ in both sides of (\ref{007}), then (1) is proved.

(2) Using Lemma \ref{lemma-3} for $U_i=\phi_i^j V_j$, (1) of Lemma \ref{lemma-5} and (\ref{101}), we get
\begin{eqnarray*}
\widetilde{U}_i&=&\phi_i^j\widetilde{V}_j+
v^j V_j\phi_i^k\partial_{v^k}=
\phi_i^j\widetilde{V}_j+u^p\phi_p^j V_j\phi_i^k\partial_{v^k}\\
&=&\phi_i^j\widetilde{V}_j+u^p c^q_{pi}
\phi_q^k\partial_{v^k}=\phi_i^j\widetilde{V}_j+ c^q_{pi}u^p \partial_{u^q},
\end{eqnarray*}
which proves (2).

(3) By (2) of Lemma \ref{lemma-5}, we have
\begin{eqnarray*}
\widetilde{U}_iu^j=\phi_i^k\widetilde{V}_k u^j
+ c^q_{pi}u^p\tfrac{\partial}{\partial u^q}u^j
=c^j_{li}u^l,
\end{eqnarray*}
which proves (3).

(4) Using (2) of Lemma \ref{lemma-5} again, we get
\begin{eqnarray*}
[\widetilde{U}_i,\partial_{u^l}]=
[\phi_i^j\widetilde{V}_j+c_{qi}^pu^q\partial_{u^p},\partial_{u^l}]
=[c_{qi}^pu^q\partial_{u^p},\partial_{u^l}]
=c_{il}^p\partial_{u^p},
\end{eqnarray*}
which ends the proof of Lemma \ref{lemma-5}.
\end{proof}

We will always keep above notations in the discussion below.

\subsection{Invariant spray on a Lie group}
\label{section-3-2}
A spray structure $\mathbf{G}$ on the   Lie group $G$ is called
{\it left invariant} (or {\it right invariant}), if $(\widetilde{L}_g)_*\mathbf{G}=\mathbf{G}$ (or $(\widetilde{R}_g)_*\mathbf{G}=\mathbf{G}$)
for all $g\in G$, or equivalently, $[\widetilde{V}_i,\mathbf{G}]=0$
(or $[\widetilde{U}_i,\mathbf{G}]=0$ respectively) for each $i$.
It is bi-invariant if it is both left and right invariant.

Theorem A claims the existence of the canonical Berwald bi-invariant
spray structure on each Lie group. In later discussion, $\mathbf{G}_0$ is always referred to this special spray structure $\mathbf{G}_0=
u^i\widetilde{U}_i=v^i\widetilde{V}_i$ on each Lie group.\bigskip


\begin{proof}[Proof of Theorem A]
Firstly, we prove that  $\mathbf{G}_0=u^i\widetilde{U}_i$ 
is a Berwald spray structure. Obviously, it is a globally defined tangent field on $TG\backslash0$.

Let $(x^i,y^i)$ be standard local coordinates on $TM$. We apply the notations in (\ref{006}), i.e., $U_i=A^j_i \partial_{x^j}$, and $(B_i^j)=(A_i^j)^{-1}$.
By Lemma \ref{lemma-1},
$\widetilde{U}_i=A^j_i\partial_{x^j}+
y^j\tfrac{\partial}{\partial x^j}A^k_i\ \partial_{y^k}$.
So we have
\begin{eqnarray}\label{005}
\mathbf{G}_0=u^i\widetilde{U}_i=
u^iA^j_i\partial_{x^j}+
u^iy^j\tfrac{\partial}{\partial x^j}A^k_i\ \partial_{y^k}=
y^j\partial_{x^j}+ y^l y^j (B^i_l\tfrac{\partial}{\partial x^j}A^k_i)\
\partial_{y^k},
\end{eqnarray}
Notice that the functions $A^i_j$ and $B^i_j$ only depend on the $x^i$-coordinates. So we see from (\ref{005}) that $\mathbf{G}_0$
is a Berwald spray structure.

Secondly, we prove $u^i\widetilde{U}_i=v^i\widetilde{V}_i$.
By (2) in Lemma \ref{lemma-5} and (\ref{101}),
\begin{eqnarray*}
u^i\widetilde{U}_i=u^i\phi_i^j\widetilde{V}_j+
u^iu^p c^q_{pi}\partial_{u^q}=
v^j\widetilde{V}_j+u^iu^p c^q_{pi}\partial_{u^q}=v^j\widetilde{V}_j,
\end{eqnarray*}
where $u^iu^p c^q_{pi}$ vanishes because $c^q_{pi}=-c^q_{ip}$.

Finally, we apply (\ref{100}) to the presentations
$\mathbf{G}_0=u^i\widetilde{U}_i=v^j\widetilde{V}_j$ and get the
bi-invariancy $[\widetilde{U}_i,\mathbf{G}_0]=[\widetilde{V}_i,\mathbf{G}_0]=0$
for each $i$.
\end{proof}
\bigskip

Then we consider any left invariant spray structure $\mathbf{G}$ on $G$.
Compare $\mathbf{G}$ with $\mathbf{G}_0=u^i\widetilde{U}_i=v^i\widetilde{V}_i$, we see
$\mathbf{H}=\mathbf{G}_0-\mathbf{G}=\mathbf{H}^i\partial_{u^i}$ is a left invariant section of $\mathcal{V}$. The left invariancy of $\mathbf{H}$ implies each $\mathbf{H}^i$ is preserved by the action of $\widetilde{L}_g$, $\forall g\in G$. Since the right invariant tangent fields generate the left translations, we have $\widetilde{V}_j\mathbf{H}^i=0 $, $\forall i,j$. Further more, every $\mathbf{H}^i$ has a degree-two positive homogeneity when restricted to each $T_gG\backslash\{0\}$.

The technique of invariant frame enable us to assemble the Lie bracket coefficients, the invariant $u^i$-coordinates and $\mathbf{H}^i$'s to
global clean curvature formulae for $\mathbf{G}$. The
calculation for S-curvature and Riemann curvature are carried out below as examples.

\subsection{left invariant S-curvature formula}
\label{section-3-3}

Now we use the left invariant frame $\{\widetilde{U}_i,\partial_{u^i},\forall i\}$ and any left invariant smooth measure to calculate the S-curvature $S$
 for the left invariant spray structure $\mathbf{G}=\mathbf{G}_0-\mathbf{H}=\mathbf{G}_0-
\mathbf{H}^i\partial_{u^i}$.

To apply the definition (\ref{118}), we need standard local coordinates $(x^i,y^i)$ on $TM$, and the corresponding notations in (\ref{006}), i.e.,  $U_i=A^j_i\partial_{x^j}$ and $u^i=B_j^i y^j$, where $(A_i^j)$ and $(B_i^j)=(A_i^j)^{-1}$ only depend on the $x^i$-coordinates.

By Lemma \ref{lemma-1} and Theorem A,
\begin{eqnarray*}
\mathbf{G}&=&u^i \widetilde{U}_i-\mathbf{H}^i\partial_{u^i}
=u^i (A_i^j \partial_{x^j}+y^j\tfrac{\partial}{\partial x^j}A^k_i \partial_{y^k})-A_j^k\mathbf{H}^j
 \partial_{y^k}\nonumber\\
&=&y^j \partial_{x^j}-(A_i^k\mathbf{H}^i-u^i y^j\tfrac{\partial}{\partial x^j}A^k_i) \partial_{y^k},
\end{eqnarray*}
so we have
\begin{equation}\label{009}
\mathbf{G}^k=\tfrac12 A_i^k\mathbf{H}^i-
\tfrac12u^iy^l\tfrac{\partial}{\partial x^l}A_i^k.
\end{equation}
Further calculation shows
\begin{equation}\label{111}
N_j^k=\tfrac{\partial}{\partial y^j}G^k=
\tfrac12 A^k_i\tfrac{\partial}{\partial y^j}\mathbf{H}^i-
\tfrac12 u^i\tfrac{\partial}{\partial x^j}A^k_i-
\tfrac12B^i_j y^l\tfrac{\partial}{\partial x^l}A^k_i,
\end{equation}
and its trace is
\begin{equation}\label{112}
N_j^j=\tfrac12 A^j_i\tfrac{\partial}{\partial y^j}\mathbf{H}^i-
\tfrac12 u^i\tfrac{\partial}{\partial x^j}A^j_i-
\tfrac12B^i_j y^l\tfrac{\partial}{\partial x^l}A^j_i.
\end{equation}

From (\ref{118}), we see the S-curvature is irrelevant to the choice of the left invariant measure $d\mu$. So we may take the left invariant volume form
$d\mu=U_1^*\wedge\cdots\wedge U_n^*$ where $\{U_i^*,\forall i\}$ is the dual frame for $\{U_i,\forall i\}$. By standard local coordinates,
$$d\mu=\sigma(x)dx^1\wedge\cdots\wedge dx^n=\det(B_i^j)dx^1\wedge\cdots\wedge dx^n,$$
so we have
\begin{eqnarray}\label{113}
-{\sigma(x)}^{-1}{y^m}\tfrac{\partial}{\partial x^m}\sigma(x)
=-y^m A_j^i\tfrac{\partial}{\partial x^m}B_i^j=
y^m B_i^j\tfrac{\partial}{\partial x^m}A^i_j.
\end{eqnarray}

Adding (\ref{112}) and (\ref{113}), we get
\begin{eqnarray}
S&=&\tfrac12A^j_i\tfrac{\partial}{\partial y^j}\mathbf{H}^i
+\tfrac12y^mB_i^j\tfrac{\partial}{\partial x^m}A_j^i-
\tfrac12 u^i\partial_{x^j}A_i^j\nonumber\\
&=&\tfrac12A^j_i\tfrac{\partial}{\partial y^j}\mathbf{H}^i
+\tfrac12B_i^j u^l A_l^m\tfrac{\partial}{\partial x^m}A_j^i
-\tfrac12 u^i\partial_{x^j}A_i^j\nonumber\\
&=&\tfrac12A^j_i\tfrac{\partial}{\partial y^j}\mathbf{H}^i
+\tfrac12 B_i^j u^l(A_l^m\tfrac{\partial}{\partial x^m}A_j^i-
A_j^m\tfrac{\partial}{\partial x^m}A_l^i)
+(\tfrac12B_i^j u^l A_j^m\tfrac{\partial}{\partial x^m}A_l^i-
\tfrac12 u^i\tfrac{\partial}{\partial{x^j}}A_i^j)\nonumber\\
&=&\tfrac12A^j_i\tfrac{\partial}{\partial y^j}\mathbf{H}^i
+\tfrac12 B_i^j u^l(A_l^m\tfrac{\partial}{\partial x^m}A_j^i-
A_j^m\tfrac{\partial}{\partial x^m}A_l^i)+
(\tfrac12 u^l \tfrac{\partial}{\partial x^i}A_l^i-
\tfrac12 u^i\tfrac{\partial}{\partial{x^j}}A_i^j)\nonumber\\
&=&\tfrac12\tfrac{\partial}{\partial u^i}\mathbf{H}^i
+\tfrac12 B_i^j u^l(A_l^m\tfrac{\partial}{\partial x^m}A_j^i-
A_j^m\tfrac{\partial}{\partial x^m}A_l^i).\label{114}
\end{eqnarray}

We compare the coefficients of $\partial_{x}^k$ in
$$c_{qi}^p A_p^k\partial_{x^k}=c_{qi}^p U_p=[U_q,U_i]=(A^j_q\tfrac{\partial}{\partial x^j}A_i^k-A^j_i\tfrac{\partial}{\partial x^j}A^k_q)\partial_{x}^k,$$
and get
\begin{equation}\label{115}
A^j_q\tfrac{\partial}{\partial x^j}A_i^k-A^j_i\tfrac{\partial}{\partial x^j}A^k_q=
c_{qi}^p A_p^k.
\end{equation}

With the help of (\ref{115}), the calculation in (\ref{114}) can be continued as
\begin{eqnarray*}
S&=&\tfrac12\tfrac{\partial}{\partial u^i}\mathbf{H}^i
+\tfrac12 B_i^j u^l(A_l^m\tfrac{\partial}{\partial x^m}A_j^i-
A_j^m\tfrac{\partial}{\partial x^m}A_l^i)\\
&=&\tfrac12\tfrac{\partial}{\partial u^i}\mathbf{H}^i
+\tfrac12 B_i^j u^l c_{lj}^p A_p^i
=\tfrac12\tfrac{\partial}{\partial u^i}\mathbf{H}^i
+\tfrac12 c_{lj}^j u^l.
\end{eqnarray*}

To summarize, we get the left invariant S-curvature formula
\begin{equation}\label{116}
S=\tfrac12 \tfrac{\partial}{\partial u^i}\mathbf{H}^i
+\tfrac12 c_{lj}^j u^l,
\end{equation}
which ends the proof of Theorem B.

%

\subsection{Left invariant Riemann curvature formula}
\label{section-3-4}
Now we use the left invariant frame $\{\widetilde{U}_i,\partial_{u^i},\forall i\}$ to calculate the Riemann curvature $R$
 for the left invariant spray structure $\mathbf{G}=\mathbf{G}_0-\mathbf{H}=\mathbf{G}_0-
\mathbf{H}^i\partial_{u^i}$.

Firstly, we observe that the horizonal distribution is linearly spanned
by the horizonal liftings $\widetilde{U}_q^{\mathcal{H}}$ for all $q$. So if the Riemann curvature $R$ is viewed as a smooth section
of $\mathcal{H}^*\otimes\mathcal{V}$ over $TG\backslash0$,
it is linearly determined by the evaluations $R(\widetilde{U}_q^\mathcal{H})$ for all $q$.
Calculation shows
\begin{lemma}\label{lemma-6}
The horizonal lifting of $U_q$ is
$\widetilde{U}_q^\mathcal{H}=\widetilde{U}_q-(\tfrac12\tfrac{\partial}{\partial u^q}\mathbf{H}^i
-\tfrac12 u^jc^i_{qj})\partial_{u^i}$.
\end{lemma}
\begin{proof}
Let $(x^i,y^i)$ be the standard local coordinates on $TM$ with
$U_q=A_q^j\partial_{x^j}$ and $u^i=y^jB^i_j$. By (\ref{111})
and (\ref{115}),
we have
\begin{eqnarray}
\widetilde{U}_q^\mathcal{H}&=&A_q^j\delta_{x^j}
=A_q^j\partial_{x^j}-A_q^j N_j^k\partial_{y^k}\nonumber\\
&=&\widetilde{U}_q-y^j\tfrac{\partial}{\partial x^j}A^l_q\partial_{y^l}-\tfrac12A_q^jA_i^k
\tfrac{\partial}{\partial y^j}\mathbf{H}^i \partial_{y^k}+\tfrac12 u^i A^j_q\tfrac{\partial}{\partial x^j}A_i^k\partial_{y^k}
+\tfrac12 y^l\tfrac{\partial}{\partial x^l}A^k_q\partial_{y^k}\nonumber\\
&=&\widetilde{U}_q-\tfrac12\tfrac{\partial}{\partial u^q}\mathbf{H}^i\partial_{u^i}+
\tfrac12u^i(A^j_q\tfrac{\partial}{\partial x^j}A_i^k-A^j_i\tfrac{\partial}{\partial x^j}A^k_q)\partial_{y^k}\nonumber\\
&=&\widetilde{U}_q-\tfrac12\tfrac{\partial}{\partial u^q}
\mathbf{H}^i\partial_{u^i}+\tfrac12 c_{qi}^l A^k_lu^i\partial_{y^k}
\nonumber\\
&=&\widetilde{U}_q-(\tfrac12\tfrac{\partial}{\partial u^q}
\mathbf{H}^i-\tfrac12 c_{qj}^i u^j)\partial_{u^i},
\label{102}
\end{eqnarray}
which ends the proof of Lemma \ref{lemma-6}.
\end{proof}

Nextly, we calculate $R(\widetilde{U}_q^{\mathcal{H}})$, i.e., the $\mathcal{V}$-summand in $[\mathbf{G},\widetilde{U}_q^\mathcal{H}]$.
Using Lemma \ref{lemma-5} and Lemma \ref{lemma-6}, we get
\begin{eqnarray}
& &[-\mathbf{H}^i\partial_{u^i},\widetilde{U}_q^\mathcal{H}]
\nonumber\\
&=&[-\mathbf{H}^i\partial_{u^i},\widetilde{U}_q
-\tfrac12\tfrac{\partial}{\partial u^q}\mathbf{H}^i\partial_{u^i}
+\tfrac12 c^i_{qj} u^j\partial_{u^i}]
\nonumber\\
&=&\widetilde{U}_q\mathbf{H}^i\partial_{u^i}
+\mathbf{H}^i[\widetilde{U}_q,\partial_{u^i}]
+\tfrac12[\mathbf{H}^i\partial_{u^i},\tfrac{\partial}{\partial u^q}\mathbf{H}^i\partial_{u^i}]-
\tfrac12 c^i_{qj}[\mathbf{H}^p\partial_{u^p},u^j\partial_{u^i}]\nonumber\\
&=&c^r_{pq}u^p\tfrac{\partial}{\partial u^r}\mathbf{H}^i\partial_{u^i}+c^i_{qp}
\mathbf{H}^p\partial_{u^i}
+\tfrac12\mathbf{H}^p\tfrac{\partial^2}{\partial u^p\partial u^q}
\mathbf{H}^i\partial_{u^i}-\tfrac12\tfrac{\partial}{\partial u^p}
\mathbf{H}^p\tfrac{\partial}{\partial u^p}\mathbf{H}^i\partial_{u^i}\nonumber\\
& &-\tfrac12c^i_{qj}\mathbf{H}^j\partial_{u^i}
+\tfrac12c^p_{qj}u^j\tfrac{\partial}{\partial u^p}\mathbf{H}^i\partial_{u^i}\nonumber\\
&=&\tfrac12 c^r_{pq}u^p\tfrac{\partial}{\partial u^r}\mathbf{H}^i\partial_{u^i}+
\tfrac12c_{qj}^i\mathbf{H}^j\partial_{u^i}
+\tfrac12\mathbf{H}^p\tfrac{\partial^2}{\partial u^p\partial u^q}
\mathbf{H}^i\partial_{u^i}-
\tfrac12\tfrac{\partial}{\partial u^q}
\mathbf{H}^p
\tfrac{\partial}{\partial u^p}\mathbf{H}^i
\partial_{u^i}.\label{103}
\end{eqnarray}
By Theorem A, $\mathbf{G}_0=u^p\widetilde{U}_p=v^p\widetilde{V}_p$ commutes with every $\widetilde{U}_q$. So we can use (4) of Lemma \ref{lemma-5} and Lemma \ref{lemma-6}
to get
\begin{eqnarray}
& &[u^p\widetilde{U}_p,\widetilde{U}_q^\mathcal{H}]
=[u^p\widetilde{U}_p,\widetilde{U}_q-
\tfrac12(\tfrac{\partial}{\partial u^q}\mathbf{H}^i
-c^i_{qj}u^j)\partial_{u^i}]\nonumber\\
&=&[u^p\widetilde{U}_p,-\tfrac12(\tfrac{\partial}{\partial u^q}
\mathbf{H}^i-c^i_{qj}u^j)\partial_{u^i}]\nonumber\\
&=&\tfrac12(\tfrac{\partial}{\partial u^q}\mathbf{H}^p
-c^p_{qj}u^j)\widetilde{U}_p-
\tfrac12u^p(\tfrac{\partial}{
\partial u^q}\mathbf{H}^i-c^i_{qj}u^j)
[\widetilde{U}_p,\partial_{u^i}]
\nonumber\\
&=&\tfrac12(\tfrac{\partial}{\partial u^q}\mathbf{H}^p
-c^p_{qj}u^j)(\widetilde{U}_p^\mathcal{H}+
\tfrac12(\tfrac{\partial}{\partial u^p}\mathbf{H}^i-c^i_{pr}u^r)\partial_{u^i})-\tfrac12 u^p
(\tfrac{\partial}{\partial u^q}\mathbf{H}^r-c^r_{qj}u^j)c^i_{pr}
\partial_{u^i}\nonumber\\
&=&\tfrac14(\tfrac{\partial}{\partial u^q}\mathbf{H}^p
-c^p_{qj}u^j)
(\tfrac{\partial}{\partial u^p}\mathbf{H}^i-c^i_{pr}u^r)\partial_{u^i}-
\tfrac12 u^p
(\tfrac{\partial}{\partial u^q}\mathbf{H}^r-c^r_{qj}u^j)c^i_{pr}
\partial_{u^i}
\quad\mbox{(mod }\mathcal{H}\mbox{)}\nonumber\\
&=&
\tfrac14\tfrac{\partial}{\partial u^q}\mathbf{H}^p
\tfrac{\partial}{\partial u^p}\mathbf{H}^i\partial_{u^i}-
\tfrac14 c^p_{qj} u^j\tfrac{\partial}{\partial u^p}
\mathbf{H}^i\partial_{u^i}+
\tfrac14c^i_{pr} u^r\tfrac{\partial}{\partial u^q}
\mathbf{H}^p\partial_{u^i}-
\tfrac14c^p_{qj}c^i_{pr}u^ju^r\partial_{u^i}.
\label{104}
\end{eqnarray}
Adding (\ref{103}) and (\ref{104}),
we get
\begin{eqnarray}
R(\widetilde{U}^\mathcal{H}_q)&=&\tfrac34 c^r_{pq}u^p
\tfrac{\partial}{\partial u^r}\mathbf{H}^i\partial_{u^i}+
\tfrac12c^i_{qj}\mathbf{H}^j\partial_{u^i}+
\tfrac12\mathbf{H}^p\tfrac{\partial^2}{\partial u^p\partial u^q}
\mathbf{H}^i\partial_{u^i}\nonumber\\& &-
\tfrac14\tfrac{\partial}{\partial u^q}
\mathbf{H}^p\tfrac{\partial}{\partial u^p}\mathbf{H}^i\partial_{u^i}+
\tfrac14 c^i_{pr}u^r\tfrac{\partial}{\partial u^q}\mathbf{H}^p
\partial_{u^i}-
\tfrac14 c^p_{qj}c^i_{pr}u^ju^r\partial_{u^i},
\label{105}
\end{eqnarray}
which ends the proof of Theorem C.

\section{Spray vector field for a left invariant spray structure}
\subsection{Spray vector field and left invariant curvature formulae}
By the left invariancy, we only need to study a left invariant spray structure $\mathbf{G}=\mathbf{G}_0-\mathbf{H}=
\mathbf{G}_0-\mathbf{H}^i\partial_{u^i}$ at $e\in G$. So the restriction of $\mathbf{H}$ to $T_eG\backslash\{0\}$
is crucial. According the
convention of L. Huang \cite{Hu2015-1},
we present it as the smooth map
$\eta:\mathfrak{g}\backslash\{0\}\rightarrow\mathfrak{g}$,
$\eta(y)=\mathbf{H}^i(e,y)e_i$, $\forall y=u^i e_i\in\mathfrak{g}\backslash\{0\}$, and call it the {\it
spray vector field} for $\mathbf{G}$.

Using the spray vector field $\eta$,
we can algebraize the curvature formulae for $\mathbf{G}$.
Take the S-curvature $S$ and Riemann curvature $R$ for example.
Their restrictions to $\mathfrak{g}\backslash\{0\}$, i.e.,
the real function $S(\cdot)$ on $\mathfrak{g}\backslash\{0\}$ and the linear maps $R_y(\cdot)$ on $\mathfrak{g}$ for every $y\in\mathfrak{g}\backslash\{0\}$ respectively, contain all information of these curvatures.

We apply the following notations to translate Theorem B and Theorem C to more familiar curvature formulae.
Denote $D\eta(y,v)$ the derivative of $\eta$ at $y\in\mathfrak{g}\backslash\{0\}$
in the direction of $v$, $N(y,v)=\tfrac12D\eta(y,v)-\tfrac12[y,v]$, and
$DN(y,v,u)$ the derivative of $N(\cdot,v):\mathfrak{g}\backslash\{0\}
\rightarrow\mathfrak{g}$ at $y$ in the direction of $u$. Then easy calculation shows

\begin{corollary}\label{cor-4}
Let $\mathbf{G}$ be the left invariant spray structure on the Lie group $G$ associated with the spray vector field $\eta:\mathfrak{g}\backslash\{0\}\rightarrow \mathfrak{g}$,
then the S-curvature and the Riemann curvature satisfy
\begin{eqnarray}
S(y)&=&\mathrm{tr}_\mathbb{R}(N(y,\cdot)+\mathrm{ad}(y)),\quad\label{117}\mbox{and}\\
R_y(v)&=&
DN(y,v,\eta(y))-N(y,N(y,v))+N(y,[y,v])-[y,N(y,v)],\label{108}
\end{eqnarray}
respectively,
for any $y\in\mathfrak{g}\backslash\{0\}$ and $v\in\mathfrak{g}$.
\end{corollary}
\begin{remark}
The curvature formulae (\ref{117}) and (\ref{108}) are just
those L. Huang found in homogeneous Finsler geometry. The
reason for these coincidences is the following.
\end{remark}

Consider a left invariant Finsler metric on $G$ which induces $\mathbf{G}$. This metric is one-to-one determined by a Minkowski norm $F$ on $\mathfrak{g}$ \cite{De2012}. Let $g_y(\cdot,\cdot)$ be
the fundamental tensor of $F$, then Theorem 3.1 in \cite{XD2015} implies the spray vector field $\eta$ that L. Huang defined for a left invariant Finsler metric \cite{Hu2015-1}, i.e.,
\begin{equation}\label{121}
g_y(\eta(y),u)=g_y(y,[u,y]),\quad\forall y\in\mathfrak{g}\backslash\{0\},
\end{equation}
coincides with the one we use here. Further more, L. Huang observed
that his {\it connection operator} $N:(\mathfrak{g}\backslash\{0\})
\times\mathfrak{g}\rightarrow\mathfrak{g}$
\cite{Hu2015-1} can be presented as $N(y,v)=\tfrac12D\eta(y,v)-\tfrac12[y,v]$, i.e., the
one in Corollary \ref{cor-4}.

\subsection{Proof of Theorem D}
\label{section-4-2}
%
Let $\mathbf{G}=\mathbf{G}_0-\mathbf{H}^i\partial_{u^i}$ be a left invariant spray structure on the Lie group $G$, and $\eta:\mathfrak{g}\backslash\{0\}\rightarrow\mathfrak{g}$ its spray vector field.
Now we prove Theorem D, i.e., the correspondence between geodesics for $\mathbf{G}$
and integral curves of $-\eta$.

Let $c(t)$ be any geodesic for $\mathbf{G}$. Using standard local coordinates $(x^i,y^i)$, it can be presented as $c(t)=(c^i(t))$ which satisfies
\begin{equation}\label{010}
\ddot{c}^i(t)=-2\mathbf{G}^i(c(t),\dot{c}(t)), \quad\forall i,
\end{equation}
where $-2\mathbf{G}^i$ is coefficient in the presentation
$\mathbf{G}=y^i\partial_{x^i}-2\mathbf{G}^i\partial_{x^i}$.

Now we switch to left invariant frame $\{\widetilde{U}_i,\partial_{u^i},\forall i\}$ with $U_i(g)=A_i^j(g)\partial_{x^j}$ and $u^i=B^i_j y^j$. We denote $\dot{c}(t)=u^i(t)U_i(c(t))$, then $\dot{c}^i(t)=A_j^i(c(t))u^j(t)$ satisfies
\begin{eqnarray}\label{119}
\ddot{c}^i(t)=\tfrac{d}{dt}(A_j^i(c(t))u^j(t))=
A_j^i(c(t))\dot{u}^j(t)+u^j(t)\dot{c}^k(t)\tfrac{\partial}{\partial x^k}A^i_j(c(t)).
\end{eqnarray}
Meanwhile, using (\ref{009}), we can change
(\ref{010}) to
\begin{equation}\label{120}
\ddot{c}^i(t)=-A^i_j(c(t))\mathbf{H}^j(c(t),\dot{c}(t))+u^j(t)\dot{c}^k(t)
\tfrac{\partial}{\partial x^k}A^i_j(c(t)).
\end{equation}
Compare (\ref{119}) and (\ref{120}), we get
\begin{equation}\label{011}
\dot{u}^j(t)+\mathbf{H}^j(c(t),\dot{c}(t))=0,\quad\forall j.
\end{equation}
By the left invariancy, $\mathbf{H}^j(c(t),\dot{c}(t))=\mathbf{H}^j(e,y(t))$, where $y(t)=(L_{c(t)^{-1}})_* (\dot{c}(t))=u^j(t) e_j$, so
 (\ref{011}) implies that
$y(t)=(L_{c(t)^{-1}})_* \dot{c}(t)=u^j(t)  e_j$ is an integral curve of
$-\eta=-\mathbf{H}|_{T_eG\backslash\{0\}}$.

Above argument provide the correspondence from (1) to (2) in Theorem D. Then we discuss its inverse.

For any smooth map $y(t):(a,b)\rightarrow\mathfrak{g}\backslash\{0\}=T_eG\backslash\{0\}$, $a<0<b$, we claim the ODE
\begin{equation}\label{012}
\dot{c}(t)=(L_{c(t)})_*y(t)
 \end{equation}
has a unique solution for $t\in(a,b)$ satisfying $c(0)=e$. To prove our claim, we only need to show the existence of $c(t)$ for $t\in(a+\epsilon,b-\epsilon)$, where $\epsilon>0$ is arbitrarily small.

 Using the existence theory for ODE and by the compactness of $[a+\epsilon,b-\epsilon]$, we can find a sufficiently small positive $\delta>0$, such that for each $t_0\in[a+\epsilon,b-\epsilon]$, the
solution $c_{t_0}(t)$ of
\begin{equation}\label{013}
\dot{c}_{t_0}(t)=(L_{c_{t_0}(t)})_*(y(t-t_0)),\quad\mbox{with}\quad c_{t_0}(0)=e
\end{equation}
uniquely exists for $t\in (-\delta,\delta)$.

Let $(a',b')$ be the maximal open sub-interval
in $(a,b)$ on which the solution $c(t)$ of (\ref{012}) with $c(0)=e$ exists. Obviously $a'$ and $b'$ exist and $a'\leq -\delta$ and $b'\geq\delta$. Now
We prove $b'> b-\epsilon$ and $a'<a+\epsilon$.
Assume conversely $b'\leq b-\epsilon$. Then we take $t_0=b'-\tfrac{\delta}{2}$ and denote $c_{t_0}(t)$ with $t\in(-\delta,\delta)$ the corresponding solution of (\ref{013}). By the left invariancy of the equation (\ref{012}), $c_1(t)=c(t_0)c_{t_0}(t-t_0)$ with $t\in (t_0-\delta,t_0+\delta)$ is a solution of (\ref{012}) satisfying $c_1(t_0)=c(t_0)$. By the uniqueness theory for ODE, the solution $c(t)$ for (\ref{012}) can be extended to $(a',b'+\tfrac{\delta}2)$. This is a contradiction with our assumption for $b'$.
Similarly, we can also prove $a'<a+\epsilon$.

To summarize, the solution $c(t)$ of (\ref{012}) with $c(0)=e$ exists uniquely on $(a+\epsilon,b-\epsilon)$ for arbitrarily small $\epsilon>0$. So the solution $c(t)$ exists uniquely on $(a,b)$, which proves our claim.

Finally,
we can apply similar argument as for the correspondence from (1) to (2) to prove $c(t)$ is a geodesic on $(G,\mathbf{G})$ when $y(t)$ is
an integral curve of $-\eta$. This ends the proof of Theorem B.

\subsection{Remarks on bi-invariant spray structure}

Apply Theorem B, Theorem C and Theorem D to the canonical bi-invariant spray structure $\mathbf{G}_0$, we get
the following immediate corollary.
\begin{corollary}\label{cor-1}
For the canonical bi-invariant spray structure $\mathbf{G}_0$ on a Lie group $G$,
the set of all maximally extended geodesics consists of $c(t)=g\cdot\exp tX$ for all $g\in G$ and $X\in\mathfrak{g}$, the S-curvature is
$S=\tfrac12c^j_{lj} u^l$,
and the Riemann curvature satisfies
$R(\widetilde{U}_q)=
-\tfrac14 c^p_{qj}c^i_{pr}u^j u^r\partial_{u^i}$ for each $q$.
\end{corollary}

Using Corollary \ref{cor-4}, the S-curvature and Riemann curvature for $\mathbf{G}_0$ can also
be presented as $S(y)=\tfrac12\mathrm{tr}_\mathbb{R}\mathrm{ad}(y)$ and
$R_y(v)=-\tfrac14[y,[y,v]]$.

From these simple observations, we see many differences between spray geometry and Finsler geometry.

Firstly, bi-invariant Finsler metrics exist
on a Lie group $G$ iff its Lie algebra is compact. The spray structure induced by a bi-invariant Finsler metric must be $\mathbf{G}_0$. Further more, bi-invariant Finsler metrics always have vanishing S-curvature and are non-negatively curved.

But spray geometry tells us different stories.
Theorem A claims at least one bi-invariant spray structure $\mathbf{G}_0$ on every Lie group. A Lie group may have many different bi-invariant spray structures (see Example 4.1.3 in \cite{Sh2001} for another bi-invariant spray structure on $\mathbb{R}^2$). Corollary \ref{cor-1} indicates the S-curvature for
$\mathbf{G}_0$ does  not vanish when $G$ is not unimodular, and $R_y(v)=-\tfrac14[y,[y,v]]$
may have negative eigenvalues and may not be as diagonalizable as in Finsler geometry.

Secondly,  for a left invariant Finsler metric on a non-Abelian nilpotent Lie group, its Ricci scalar (i.e. the trace of the Riemann curvature) must be somewhere positive and somewhere negative \cite{Hu2015-2,Wo1964}. However, Corollary \ref{cor-4} tells us $\mathbf{G}_0$ is flat (i.e., the Riemann curvature vanishes identically), when $G$ is a two-step nilpotent Lie group (a Heisenberg group for example).

Finally, homogeneous Finsler spaces are all geodesically complete.
However, we can construct many incomplete bi-invariant spray structures. For example,
$$\mathbf{G}=y^1\partial_{x^1}+\cdots+y^n\partial_{x^n}-
\sqrt{(y^1)^2+\cdots+(y^n)^2}(y^1\partial_{y^1}+\cdots+
y^n\partial_{y^n})$$
is an incomplete bi-invariant spray structure on $\mathbb{R}^n$. Notice that any maximally extended integral curve $y(t)$ for $$-\eta= \sqrt{(y^1)^2+\cdots+(y^n)^2}(y^1\partial_{y^1}+\cdots+
y^n\partial_{y^n})$$
is a ray initiating from the origin, corresponding to $t\in(-\infty,C)$ for some $C\in\mathbb{R}$. By Theorem D, this spray structure is not complete.

\subsection{Application to a special case of Landsberg Conjecture}
Recently, B. Najafi and A. Tayebi proved the special case of
Landsberg Conjecture for two-dimensional homogeneous Finsler manifolds (i.e., homogeneous Finsler surfaces). Their theorem
claims
\begin{theorem}\label{thm-4}
Any homogeneous Landsberg surface is Riemannian or locally Minkowskian.
\end{theorem}

Here we can use Theorem D to give an alternative proof.
\bigskip

\begin{proof}Without loss of generality, we may assume the
homogeneous Finsler surface $(M,F)$ is connected. Denote
$G=I_0(M,F)$ the connected isometry group of $(M,F)$.
Its dimension $\dim G$ may only be 2 or 3. When $\dim G=3$,
$(M,F)$ is a Riemannian surface with constant curvature. When $\dim G=2$,
$F$ can be transferred to a left invariant Landsberg metric on $G$ which is locally isometric to $(M,F)$. When $G$ is Abelian, any left invariant Finsler metric on $G$ is locally Minkowskian.

To summarize, to prove Theorem \ref{thm-4}, we only need to prove any left invariant Landsberg metric $F$ on a two-dimensional non-Abelian connected Lie group $G$ is Riemannian.

For simplicity, we use the same $F$ to denote the Minkowski norm
the metric determines on $\mathfrak{g}$, and $\eta$ the spray vector field determined by (\ref{121}). Notice that  $\eta$ is tangent to the indicatrix $S_F=\{y|F(y)=1\}\subset\mathfrak{g}$.
We can find a basis $\{e_1,e_2\}$ for the Lie algebra $\mathfrak{g}$ of $G$ satisfying $[e_1,e_2]=e_2$ and denote $(g_{ij}(y))$ the Hessian with respect to the linear coordinates $y=y^ie_i$.  So any point $y\in S_F$
satisfies $\eta(y)=0$ iff
$$g_y(y,\mathbb{R}e_2)=g_y(y,[\mathfrak{g},y])
=g_y(\eta(y),\mathfrak{g})=0.$$
By the strong convexity of $F$, there exists exactly two points $y'$ and $y''$ where $\eta$ vanishes. The complement $S_F\backslash\{y',y''\}$ is the disjoint union of two integral curves
$y_1(t)$ and $y_2(t)$ for $-\eta$, with $t\in(-\infty,\infty)$.
By Theorem D, there exist $F$-unit speed geodesics $c_i(t)$ with $c_i(0)=e$ and $\dot{c}_i(t)=(L_{c_i(t)})_* (y_i(t))$ for all $t\in(-\infty,\infty)$.
Let $w_i(t)\in\mathfrak{g}=T_eG$ be the smooth tangent field along $y_i(t)$ (i.e., $g_{y_i(t)}(y_i(t),w_i(t))=0$) satisfying $g_{y_i(t)}(w_i(t),w_i(t))=1$ for all $t$.
After possibly replacing $w_1(t)$ with $-w_1(t)$, we may further assume that $w_1(t)$ and $w_2(t)$ can be smoothly glued and extended to the whole indicatrix $S_F$.

Notice the left translations are isometries, so they preserve fundamental tensors. The tangent field $\dot{c}_i(t)=(L_{c_i(t)})_*(y_i(t))$ is linearly parallel along the geodesic $c_i(t)$. So the speciality of dimension two and
Lemma 5.3.1 in \cite{Sh2003} tell us for each $i=1,2$,
$W_i(t)=(L_{c_i(t)})_*(w_i(t))$ is a linearly parallel vector field
along the geodesic $c_i(t)$.  The Landsberg property indicates the Cartan tensor
of the left invariant metric $F$ satisfies
$C_{\dot{c}_i(t)}(W_i(t),W_i(t),W_i(t))\equiv C_i$,
$t\in(-\infty,\infty)$, for $i=1,2$ respectively \cite{Sh2003}.
Using left translations again, we see
$C_{y_i(t)}(w_i(t),w_i(t),w_i(t))\equiv C_i$, $t\in(-\infty,\infty)$, for
$i=1,2$ respectively, and $C_1= C_2$ by the continuity when $t\rightarrow\pm\infty$.

To summarize,
we have $|C_y(w,w,w)|= C=|C_i|$ for every pair $(y,w)$ satisfying $y\in S_F$, $g_y(y,w)=0$ and $g_y(w,w)=1$. On the other hand, the speciality of dimension two implies that $C_y(w,w,w)$ coincides with the mean Cartan tensor $I(w)$, i.e.,
the derivative in the direction of $w$ for the function $f(y)=\tfrac12\ln\det(g_{ij}(y))$. At $y\in S_F$ where $f$ achieves its maximum or minimum, $C=C_y(w,w,w)=I(w)=0$. For dimension two, this is enough for us to see that the Cartan tensor vanishes identically everywhere, i.e., the left invariant metric $F$ on $G$ is Riemannian.
\end{proof}
\bigskip

\noindent
{\bf Acknowledgement}. The author sincerely thank B. Najafi, A. Tayebi, Ju Tan and Huaifu Liu for helpful discussions.
This paper is supported by Beijing Natural Science Foundation
(No.~Z180004), National Natural Science
Foundation of China (No.~11771331, No.~11821101),
and Capacity Building for Sci-Tech  Innovation -- Fundamental Scientific Research Funds (No. KM201910028021).


\begin{thebibliography}{99}
\bibitem{Be1926} L. Berwald, Untersuchung der Kr\"{u}mmung allgemeiner metrischer R\"{a}ume auf Grund des in ihnen herrschenden Parallelismus, Math. Z. {\bf25} (1926), 40-73.

\bibitem{BCS2000} D. Bao, S.S. Chern and Z. Shen, An introduction to Riemann-Finsler geometry, G.T.M. {\bf200}, Springer-Verlag, New York, (2000).

\bibitem{De2012} S. Deng, Homogeneous Finsler spaces, Springer, (2012).

\bibitem{DZ1979} J. D'Atri and W. Ziller, Naturally reductive metrics and Einstein metrics on compact Lie groups, Mem. Amer. Math. Soc. {\bf18} (215), 1-72.

\bibitem{Hu2015-1} L. Huang, On the fundamental equations of  homogeneous Finsler spaces, Diff. Geom. Appl. {\bf40} (2015),
    187-208.

\bibitem{Hu2015-2} L. Huang, Ricci curvatures of left invariant Finsler metrics on Lie groups, Israel J. Math. {\bf207} (2) (2015), 783-792.

\bibitem{Hu2017}L. Huang, Flag curvatures of homogeneous Finsler spaces, Euro. J. Math., S.I. {\bf3} (4) (2017), 1000-1029.

\bibitem{LS2017} B. Li and Z. Shen, Sprays of isotropic curvature, Int. J. Math. {\bf 29} (1) (2018), 1850003.

\bibitem{Mi1976} J. Milnor, Curvatures of left invariant metrics on Lie groups, Adv. Math. {\bf21} (3) (1976), 293-329.

\bibitem{Ma1996} M. Matsumoto, Remarks on Berwald and Landsberg spaces, Finsler geometry (Seattle, WA, 1995), 79-82, Contemp. Math. {\bf 196}, Amer. Math. Soc., Providence, RI, 1996.

\bibitem{NT2021} B. Najafi and A. Tayebi, On homogeneous Landsberg surfaces, preprint, (2021).

\bibitem{Sh1997} Z. Shen, Volume comparison and its applications in Riemann-Finsler geometry, Adv. Math. {\bf128}
    (2) (1997), 306-328.

\bibitem{Sh2001} Z. Shen, Differential Geometry of Spray and Finsler Spaces, Kluwer Academic Publishers, (2001).

\bibitem{Sh2003} Z. Shen, Lectures on Finsler Geometry, World Scientific, (2001).

\bibitem{Wo1964} J.A. Wolf, Curvature in nilpotent Lie groups, Proc. Amer. Math. Soc. {\bf15} (1964), 271-274.

\bibitem{XD2015} M. Xu and S. Deng, Killing frames and S-curvature of homogeneous Finsler spaces, Glasg. Math. J.
    {\bf57} (2) (2015), 457-464.

\bibitem{XM2020} M. Xu and V.S. Matveev, Proof of Laugwitz Conjecture and Landsberg Unicorn Conjecture for Minkowski norms with $SO(k)\times SO(n-k)$-symmetry, preprint (2020), arXiv:arXiv:2007.15888.

\bibitem{Ya2011} G. Yang, Some classes of sprays in projective spray geometry, Diff. Geom. Appl. {\bf29} (2011), 606-614.

\bibitem{YD2017} Z. Yan and S. Deng, Einstein metrics on compact simple Lie groups attached to standard triples,
    Trans. Amer. Math. Soc. {\bf369} (12) (2017), 8587-8605.

\end{thebibliography}
\end{document}